\documentclass[12pt]{article}
\usepackage{graphicx,url}
\usepackage{amssymb,amsmath, amsthm}
\usepackage{epstopdf}
\numberwithin{equation}{section}
\newtheorem{theorem}{Theorem}

\def\de{{\rm d}}

\title{Adaptive LASSO-type estimation for ergodic diffusion processes}
\author{A. De Gregorio\\
Dipartimento di Statistica, Probabilit\`a e Statistiche Applicate,\\
P.le Aldo Moro 5, 00185- Rome, Italy \\
alessandro.degregorio@uniroma1.it
 \and
 S.M. Iacus\\
 Dipartimento di Scienze Economiche, Aziendali e Statistiche,\\
 Via Conservatorio 22, 20122 - Milan, Italy \\
  stefano.iacus@unimi.it
 }
\begin{document}
\maketitle

\begin{abstract}
The LASSO is a widely used statistical methodology for simultaneous estimation and variable selection. In the last years, many authors analyzed this technique from a theoretical and applied point of view. We introduce and study the adaptive LASSO problem for discretely observed ergodic diffusion processes. We prove oracle properties also deriving the asymptotic distribution of the LASSO estimator. Our theoretical framework is based on the random field approach and it applied to more general families of regular statistical experiments in the sense of  Ibragimov-Hasminskii (1981). Furthermore, we perform a simulation and real data analysis to provide some evidence on the applicability of this method.  \\

\end{abstract}

\section{Introduction}

The least absolute shrinkage and selection operator (LASSO) is a useful and well studied approach to the problem of model selection and its major advantage is the simultaneous execution of both parameter estimation and variable selection (see Tibshirani, 1996; Knight and Fu, 2000, Efron {\it et al.}, 2004). 
This is realized by the fact that the dimension of the parameter space does not change (while it does with the information criteria approach, e.g. in AIC, BIC, etc), because the LASSO method only sets some parameters to zero to eliminate them from the model. The LASSO method usually consists in the minimization of an $L^2$ norm under  $L^1$ norm constraints on the parameters. Thus it usually implies least squares or maximum likelihood approach plus constraints. The important property stating that
 the correct parameters are set to zero by LASSO method under the true data generating model, is called oracle property (Fan and Li, 2001). As shown by Zou (2006), since the classical LASSO estimator uses the same amount of shrinkage for each parameters, the resulting model selection could be inconsistent. To overcome this drawback, it is possible to consider an adaptive amount of shrinkage for each parameters (Zou, 2006).

 Originally, the LASSO procedure was introduced for linear regression problems, but, in the recent years, this approach has been applied to time series analysis by several authors mainly in the case of autoregressive models. For example, just to mention a few, Wang {\it et al.} (2007) consider the problem of shrinkage estimation of regressive and autoregressive coefficients, while Nardi and Rinaldo (2008) consider penalized order selection in an AR($p$) model. The VAR case was considered in Hsu {\it et al.} (2007).
Very recently Caner (2009) studied the LASSO method for general GMM estimator also in the case of time series and Knight (2008) extended the LASSO approach to nearly singular designs.

In this paper we consider the LASSO approach for discretely observed diffusion processes. In this case, the likelihood function is not usually known in closed form, moreover most models used in application are not necessarily linear. In this paper, instead of working on a single approximation of the likelihood, we study the problem in terms of random fields (see Yoshida, 2005) which encompasses all widely used methods in the literature of inference for discretely sampled diffusion processes.
Although we do not explicitly state the results in this form, the proofs in this paper, based on the properties of random fields, are immediately extensible to regular statistical experiments in the sense of Ibragimov-Hasmkinskii (1981), i.e. they apply to i.i.d. as well as regressive and autoregressive models.

For diffusion processes, the LASSO method requires some additional care because the rate of convergence of the parameters in the drift and the diffusion coefficient are different. We point out that, the usual model selection strategy based on AIC (see Uchida and Yoshida, 2005) usually depends on the properties of the estimators but also  on the method used to approximate the likelihood. Indeed, AIC requires the calculation of the likelihood (see Iacus, 2008). On the contrary,  the present LASSO approach depends solely on the properties of the estimator and so the problem of likelihood approximation is not particularly compelling.

It is worth to mention that, model selection for continuous time diffusion processes was considered earlier in Uchida and Yoshida (2001) by means of information criteria.

The paper is organized as follows. Section \ref{sec1} introduced the model and the regularity assumptions and states the problem of LASSO estimation for discretely sampled diffusion processes.
Section \ref{sec2} proves consistency and oracle properties of the LASSO estimator. 
Section \ref{sec3} contains a Monte Carlo analysis and one application to real financial data.
Proofs are collected in Section \ref{sec4}. Tables and figures at the end of the manuscript.

\section{The LASSO problem for diffusion models}\label{sec1}

In the first part of this Section, we introduce the model on which makes inference and some basic notations. Let $X_t,t>0,$ be a $d$-dimensional diffusion process solution of the following stochastic differential equation
\begin{equation}\label{eq:sde}
\de X_t = b(\alpha, X_t) \de t + \sigma(\beta,X_t)  \de W_t
\end{equation}
 where $\alpha=(\alpha_1,...,\alpha_{p})\in\Theta_p\subset \mathbb{R}^p$, $p\geq1$, $\beta=(\beta_1,...,\beta_q)\in\Theta_q\subset \mathbb{R}^q$, $q\geq1$, $b:\Theta_p\times\mathbb{R}^d\to \mathbb{R}^d$, $\sigma:\Theta_q\times \mathbb{R}^d\to \mathbb{R}^d\times \mathbb{R}^d$ and $W_t$ is a standard Brownian motion in $\mathbb{R}^d$. We assume that the functions $b$ and $\sigma$ are known up to  the parameters $\alpha$ and $\beta$. We denote by $\theta=(\alpha,\beta)\in\Theta_p\times \Theta_q=\Theta$ the parametric vector and with $\theta_0=(\alpha_0,\beta_0)$ its unknown true value. For a matrix $A$, we denote by $A^{\otimes 2}=AA'$ and by $A^{-1}$ the inverse of $A$. Let $\Sigma(\beta,x)=\sigma(\beta,x)^{\otimes 2}$. The sample path of $X_t$ is observed only at $n+1$ equidistant discrete times $t_i$, such that $t_i-t_{i-1}=\Delta_n<\infty$ for $1\leq i\leq n $ (with $t_0=0$ and $t_{n+1}=t$). We denote by ${\bf X}_n=\{X_{t_i}\}_{0\leq i\leq n}$ our random sample with values in $\mathbb{R}^{n\times d}$. 
 
 The asymptotic scheme adopted in this paper is the following: $n\Delta_n\to \infty$, $\Delta_n\to 0$ and $n\Delta_n^2\to 0$ as $n\to \infty$. This asymptotic framework is called rapidly increasing design and the condition $n\Delta_n^2\to 0$ means that $\Delta_n$ shrinks to zero slowly.
  We need  some assumptions on the regularity of the process:
\begin{itemize}
\item[$\mathcal A_1.$]  There exists a constant $C$ such that
$$|b(\alpha_0,x)-b(\alpha_0,y)|+|\sigma(\beta_0,x)-\sigma(\beta_0,y)|\leq C|x-y|.$$

\item[$\mathcal A_2.$] $\inf_{\beta,x}\mathrm{det}(\Sigma(\beta,x))>0$.

\item[$\mathcal A_3.$] The process $X$ is ergodic for every $\theta$ with
invariant probability measure $\mu_\theta$. 

\item[$\mathcal A_4.$]  For all $m\geq 0$ and for all $\theta$, $\sup_t
E|X_t|^m<\infty$.

\item[$\mathcal A_5.$] For every $\theta$, the coefficients $b(\alpha,x)$ and
$\sigma(\beta,x)$ are five times differentiable with respect to $x$ and
the derivatives are bounded by a polynomial function in $x$, uniformly in
$\theta$. 

\item[$\mathcal A_6.$]  The coefficients $b(\alpha,x)$ and $\sigma(\beta,x)$
and all their partial derivatives respect to $x$ up to order 2 are
three times differentiable with respect to $\theta$ for all $x$ in the
state space. All derivatives with respect to $\theta$ are bounded by a polynomial
function in $x$, uniformly in $\theta$.
\item[$\mathcal A_7.$] If the coefficients $b(\alpha,x)=b(\alpha_0,x)$ and
$\sigma(\beta,x)=\sigma(\beta_0,x)$  for all
$x$  ($\mu_{\theta_0}$-almost surely), then $\alpha=\alpha_0$ and $\beta=\beta_0$.
\end{itemize}
Hereafter, we assume that the conditions $\mathcal{A}_1-\mathcal{A}_7 $ hold. Let
$\mathcal I(\theta)$ be the positive definite and invertible
Fisher information matrix at $\theta$  given by
$$\mathcal I(\theta)=\left(%
\begin{array}{cc}
 \Gamma_\alpha= [\mathcal I_b^{kj}(\alpha)]_{k,j=1,...,p} & 0 \\
  0 & \Gamma_\beta= [\mathcal I_\sigma^{kj}(\beta)]_{k,j=1,...,q} \\
\end{array}%
\right)$$ where
$$\mathcal I_b^{kj}(\alpha)=\int\frac{1}{\sigma^2(\beta,x)}\frac{\partial b(\alpha,x)}{\partial\alpha_k}\frac{\partial b(\alpha,x)}{\partial\alpha_j}
\mu_{\theta}(dx)\,,$$
$$\mathcal I_\sigma^{kj}(\beta)=2\int\frac{1}{\sigma^2(\beta,x)}\frac{\partial \sigma(\beta,x)}{\partial\beta_k}\frac{\partial \sigma(\beta,x)}{\partial\beta_j}
\mu_{\theta}(dx)\,.$$ 
Moreover, we consider the matrix
$$\varphi(n)=\left(%
\begin{array}{cc}
  \frac{1}{n{\Delta_n}}{\bf I}_p& 0 \\
  0 &  \frac{1}{n}{\bf I}_q\\
\end{array}%
\right)$$
where ${\bf I}_p$ and ${\bf I}_q$ are respectively the indentity matrix of order $p$ and $q$.
 
In order to introduce the LASSO problem, we consider a random field $\mathbb{H}_n:\mathbb{R}^{n\times d}\times \Theta\to \mathbb{R}$ admitting the first and second derivatives with respect to $\theta$; we denote by $\dot{\mathbb{H}}_n({\bf X}_n, \theta)$ the vector of the first derivatives and by $\ddot{\mathbb{H}}_n({\bf X}_n, \theta)$ the Hessian matrix. Furthermore, we assume that the following conditions hold:
\begin{itemize}
\item[$\mathcal B_1.$] for each $\theta\in \Theta$, we have that
\begin{equation}
\label{eq31} \varphi(n)^{1/2}\ddot{\mathbb{H}}_n({\bf X}_n, \theta)
\varphi(n)^{1/2}\stackrel{p}{\to}\mathcal{I}(\theta)
\end{equation}
\item[$\mathcal B_2.$]
for each $\theta\in \Theta$, let $\tilde\theta_n:\mathbb{R}^{n\times d}\to \Theta$ be a consistent  estimator of $\theta$ given by
$$
\tilde{\theta}_n=\arg\min_\theta \mathbb{H}_n({\bf X}_n,\theta)
$$
such that
\begin{equation}
\varphi(n)^{-1/2}(\tilde\theta_n-\theta)\stackrel{d}{\to}N(0,\mathcal{I}(\theta)^{-1})
\end{equation}
\end{itemize}
An example of random field (contrast function) satisfying the assumptions $\mathcal{B}_1-\mathcal{B}_2$ is given by the quasi-likelihood function $\mathbb{H}_n({\bf X}_n,\theta)=l_n({\bf X}_n,\theta)$ obtained by means the Euler approximation  (see Kessler, 1997, Yoshida, 2005), that is
\begin{eqnarray}\label{qlik}
l_n({\bf X}_n,\theta)
=\frac12\sum_{i=1}^n\left\{\log\text{det}(\Sigma_{i-1}(\beta))
+\frac{1}{\Delta_n}
\Sigma_{i-1}^{-1}(\beta)[\Delta X_i-\Delta_n b_{i-1}(\alpha)]^{\otimes 2}\right\}
\end{eqnarray}
where $\Delta X_i=X_{t_i}-X_{t_{i-1}}$, $\Sigma_i(\beta)=\Sigma(\beta,X_{t_i})$ and $b_i(\alpha)=b(\alpha,X_{t_i})$.  Then the unpenalized estimator 

$$\tilde\theta_n=\arg\min_\theta l_n({\bf X}_n,\theta)$$
satisfies the assumption $\mathcal{B}_2$. For other examples, the reader can consult Bibby amd Sorensen, (1995), Kessler and Sorensen (1999), Nicolau (2002) and A\"\i t-Sahalia (2008).

The classical adaptive LASSO objective function, in this case, should be given by
\begin{equation}\label{eq:stlasso}
\mathbb{H}_n({\bf X}_n,\theta)+\sum_{j=1}^p\lambda_{n,j}|\alpha_j| +\sum_{k=1}^q\gamma_{n,k}|\beta_k|
\end{equation}
where $\lambda_{n,j}$ and $\gamma_{n,k}$ assume real positive values representing an adaptive amount of the shrinkage for  each elements of $\alpha$ and $\beta$.
Nevertheless, following the same approach of Wang and Leng (2007), we observe that by means of a Taylor expansion of $\mathbb{H}_n({\bf X}_n,\theta)$ at $\tilde{\theta}_n$, one has immediately that
\begin{eqnarray*}
\mathbb{H}_n({\bf X}_n,\theta)&=&\mathbb{H}_n({\bf X}_n,\tilde{\theta}_n)+\dot{\mathbb{H}}_n({\bf X}_n,\tilde{\theta})(\theta-\tilde{\theta}_n)'+\frac12(\theta-\tilde{\theta}_n)\ddot{\mathbb{H}}_n({\bf X}_n, \tilde\theta_n)(\theta-\tilde{\theta}_n)'+o_p(1)\\
&=&\mathbb{H}_n({\bf X}_n,\tilde{\theta}_n)+\frac12(\theta-\tilde{\theta}_n)\ddot{\mathbb{H}}_n({\bf X}_n, \tilde\theta_n)(\theta-\tilde{\theta}_n)'+o_p(1)
\end{eqnarray*}
Therefore, we use the following objective function
\begin{equation}\label{eq:lassotype}
\mathcal{F}(\theta)=(\theta-\tilde{\theta}_n)\ddot{\mathbb{H}}_n({\bf X}_n, \tilde\theta_n)(\theta-\tilde{\theta}_n)'+\sum_{j=1}^p\lambda_{n,j}|\alpha_j| +\sum_{k=1}^q\gamma_{n,k}|\beta_k|
\end{equation}
instead of \eqref{eq:stlasso}, and the LASSO-type estimator $\hat \theta_n:\mathbb{R}^{n\times d}\to \Theta$ is defined as
\begin{equation}
\hat{\theta}_n=(\hat\alpha_n,\hat\beta_n)=\arg\min_\theta\mathcal{F}(\theta).
\end{equation}
 The function $\mathcal{F}(\theta)$ is a penalized quadratic form and it has the advantage to provide an unified theoretical framework. Indeed, the objective function \eqref{eq:stlasso} allows us to perform correctly the LASSO procedure only if $\mathbb{H}_n$ is strictly convex and this fact restricts the choice of the possible contrast functions for the model \eqref{eq:sde}. Then, the function \eqref{eq:lassotype} overcomes this criticality. We also point out that $\mathcal{F}(\theta)$ has two constraints, because the drift and diffusion  parameters $\alpha_j$ and $\beta_k$ are well separated with different rates of convergence.

\section{Oracle properties}\label{sec2}
As observed by Fan and Li (2001), a good procedure should have the oracle properties, that is:
\begin{itemize}
\item identifies the right subset model;
\item has the optimal estimation rate and converge to a Gaussian random variable $N(0,\Sigma)$ where $\Sigma$ is the covariance matrix of the true subset model.
\end{itemize} 
The aim of this Section is to prove that LASSO-type estimator $\hat\theta_n$ has a good behavior in the oracle sense.

As shown by Zou (2006) the classical LASSO estimation cannot be as efficient as the oracle and the selection  results could be inconsistent, whereas its adaptive version has the oracle properties. Without loss of generality, we assume that the true model, indicated by $\theta_0=(\alpha_0,\beta_0)$, has parameters $\alpha_{0j}$ and $\beta_{0k}$ equal to zero for $p_0<j\leq p$ and $q_0<k\leq q$, while $\alpha_{0j}\neq0$ and $\beta_{0k}\neq0$ for $1\leq j\leq p_0$ and $1\leq k\leq q_0$. To study the asymptotic properties of the LASSO-type estimator $\hat\theta_n$,
we consider the following conditions:
\begin{itemize}
\item[$\mathcal{C}_1$.]  $\frac{\mu_n}{\sqrt{n \Delta_n}}\to 0$ and $\frac{\nu_n}{\sqrt{n}}\to 0$ where $\mu_n=\max\{\lambda_{n,j},1\leq j\leq p_0\}$ and $\nu_n=\max\{\gamma_{n,k},1\leq k\leq q_0\}$
\item[$\mathcal{C}_2$.] $\frac{\kappa_n}{\sqrt{n\Delta_n}}\to \infty$ and $\frac{\omega_n}{\sqrt{n}}\to \infty$ where $\kappa_n=\min\{\lambda_{n,j},j> p_0\}$ and $\omega_n=\min\{\gamma_{n,k},k> q_0\}$
\end{itemize}

The assumption $\mathcal{C}_1$ says us that the maximal tuning coefficient for the parameter $\alpha_j$ and $\beta_k$, with $1 \leq j\leq p_0$ and $1\leq k\leq q_0$, tends to zero faster than $(n\Delta_n)^{-\frac12}$ and $n^{-\frac12}$ respectively and then implies that $\sqrt{n\Delta_n}\mu_n\to0$, $\sqrt{n}\nu_n\to0$. Analogously, we observe that $\mathcal{C}_2$ means that that the minimal tuning coefficient for the parameter $\alpha_j$ and $\beta_k$, with $j> p_0$ and $ k> q_0$, tends to infinite faster than $\sqrt{n\Delta_n}$ and $\sqrt{n}$.
 
\begin{theorem} 
Under the conditions $\mathcal{B}_1,\,\mathcal{B}_2$ and $\mathcal{C}_1$, one has that 
$$\hat{\theta}_n\stackrel{p}{\to} \theta_0$$
\label{th1}
\end{theorem}

For the sake of simplicity, we denote by $\theta^*=(\alpha^*,\beta^*)$  the vector corresponding to the nonzero parameters, where $\alpha^*=(\alpha_1,...,\alpha_{p_0})$ and $\beta^*=(\beta_1,...,\beta_{q_0+1})$, while $\theta^\circ=(\alpha^\circ,\beta^\circ)'$ is the vector corresponding to the zero parameters where $\alpha^\circ=(\alpha_{p_0+1},...,\alpha_p)$ and $\beta^\circ=(\beta_{q_0+1},...,\beta_q)$. Therefore, $\theta_0=(\alpha_0,\beta_0)=(\alpha_0^*,\alpha_0^\circ,\beta_0^*,\beta_0^\circ)$ and $\hat\theta_n=(\hat\alpha_n^*,\hat\alpha_n^\circ,\hat\beta_n^*,\hat\beta_n^\circ)$.

\begin{theorem}
Under the conditions $\mathcal{B}_1,\,\mathcal{B}_2$ and $\mathcal{C}_2$,
we have that
\begin{align}
&P(\hat\alpha_n^\circ=0)\to 1\qquad \text{ and } \qquad  P(\hat\beta_n^\circ=0)\to 1.
\end{align}
\label{th2}
\end{theorem}

From Theorem \ref{th1}, we can conclude that the estimator $\hat\theta_n$ is consistent. Furthemore, Theorem \ref{th2} says us that all the estimates of the zero parameters are correctly set equal to zero with probability tending to 1. In other words, the model selection procedure is consistent and the true subset model is correctly indentified with probability tending to 1.

To complete our program, we derive the asymptotic distribution of $\hat\theta_n^*$. Hence, we indicate by
$\mathcal I_0(\theta_0^*)$ the $(p_0+q_0)\times(p_0+ q_0)$ submatrix of $\mathcal{I}(\theta)$ at point $\theta_0^*$, that is
$$\mathcal I_0(\theta_0^*)=\left(%
\begin{array}{cc}\Gamma_\alpha^{**}=
  [\mathcal I_b^{kj}(\alpha_0^*)]_{k,j=1,...,p_0} & 0 \\
  0 & \Gamma_\beta^{**}=[\mathcal I_\sigma^{kj}(\beta_0^*)]_{k,j=1,...,q_0} \\
\end{array}%
\right)$$
and introduce the following rate of convergence matrix
$$\varphi_0(n)=\left(%
\begin{array}{cc}
  \frac{1}{n{\Delta_n}}{\bf I}_{p_0}& 0 \\
  0 &  \frac{1}{n}{\bf I}_{q_0}\\
\end{array}%
\right)$$

The next result establishes that the estimator $\hat\theta_n^*$ is efficient as well as the oracle estimator.

\begin{theorem}[Oracle property]
Under the conditions $\mathcal{B}_1,\,\mathcal{B}_2$, $\mathcal{C}_1$ and $\mathcal{C}_2$, we have that
\begin{equation}\label{oracle}
\varphi_0(n)^{-\frac12}(\hat\theta_n^*-\theta_0^*)\stackrel{d}{\to} N(0,\mathcal{I}_0^{-1}(\theta_0^*))
\end{equation}
\label{th3}
\end{theorem}

Clearly, the theoretical and practical implications of our method rely to the specification of the tuning parameter $\lambda_{n,j}$ and $\gamma_{n,k}$. As observed in Wang and Leng (2007), these values could be obtained by means of some model selection criteria like generalized cross-validation, Akaike information criteria or Bayes information criteria. Unfortunately, this solution is computationally heavy and then impracticable. Therefore,
the tuning parameters should be chosen as is Zou (2006) in the following way
\begin{equation}
\label{eq:penalty}
\lambda_{n,j} = \lambda_0 |\tilde \alpha_{n,j}|^{-\delta_1}, \qquad
\gamma_{n,k} = \gamma_0 |\tilde \beta_{n,j}|^{-\delta_2}
\end{equation}
where $\tilde \alpha_{n,j}$ and  $\tilde \beta_{n,k}$ are the unpenalized estimator of $\alpha_j$ and $\beta_k$ respectively, $\delta_1, \delta_2>0$ and usually taken unitary.
The asymptotic results hold under the additional conditions
$$
\sqrt{n\Delta_n} \lambda_0\to 0, \quad (n\Delta_n)^\frac{1+\delta_1}{2}\lambda_0\to\infty, \quad\text{and}\quad
\sqrt{n}\gamma_0 \to 0, \quad n^\frac{1+\delta_2}{2}\gamma_0 \to\infty.
$$

\section{Performance of the LASSO method for small sample size}\label{sec3}
In this section we perform a small Monte Carlo analysis to check whether the LASSO method is able to select a specified model also in small samples. We also apply the method to a benchmark data set often used in the literature of model selection. The asymptotic framework of this paper is not completely realized in the next two applications, but nevertheless we test what happens outside the theoretical framework.

In both cases, we do not pretend to give extensive analysis of the method, because the previous theorems already prove the asymptotic validity of the LASSO approach for diffusion processes. Instead, we just want to show some evidence on simulated and real data to give the feeling of the applicability of the method.

\subsection{A simulation experiment}
We reproduce the experimental design in Uchida and Yoshida (2005). Therefore, we consider a diffusion process solution of the following stochastic differential equation
$$
\de X_t = -(X_t-10)\de t + 2 \sqrt{X_t}\de W_t, \quad X_0=10\,.
$$
We simulate 1000 trajectories of this process using the second Milstein scheme, i.e. the data are simulated according to
$$
\begin{aligned}
X_{t_{i+1}} =&\, \,X_{t_i} + \left(b - \frac12 \sigma \sigma_x \right) \Delta_n + \sigma Z \sqrt{\Delta_n} + \frac12 \sigma \sigma_x \Delta_n Z^2\\
&+ \Delta_n^\frac32 \left(\frac12 b \sigma_x + \frac12 b_x\sigma+\frac14 \sigma^2 \sigma_{xx}\right)Z
+ \Delta_n^2 \left(\frac12 b b_x + \frac14 b_{xx}\sigma^2\right)
\end{aligned}
$$
with $Z\sim N(0,1)$, $b_x$ and $b_{xx}$ (resp. $\sigma_x$ and $\sigma_{xx}$) are the first and second partial derivative in $x$ of the drift (resp. diffusion) coefficients (see, Milstein, 1978). This scheme has weak second-order convergence and guarantees good numerical stability.
Data are simulated at high frequency and resampled  at lower frequency $\Delta_n=0.1$ for a total of $n=1000$ observations. The simulations are done using the \texttt{sde} package  (see Iacus, 2008) for the \textsf{R} statistical environment.
So we estimate via LASSO the following five dimensional parametric model
$$
\de X_t = -\theta_1(X_t-\theta_2)\de t + (\theta_3+\theta_4 X_t)^{\theta_5}\de W_t
$$
and the true model is $(\theta_1=1, \theta_2=10, \theta_3=0, \theta_4=4, \theta_5=0.5)$. The LASSO estimator is obtained plugging in the objective function $\mathcal{F}$, the quasi-likelihood estimator and the Hessian matrix obtained by the function $\eqref{qlik}$ particularized for the present model $X_t$. 
For the penalization term we use $\lambda_0=\gamma_0 = 1$ in   \eqref{eq:penalty}.
\begin{center}
\fbox{\bf Figure \ref{fig1} about here}
\end{center}
Figure \ref{fig1} reports the density estimation of the estimates of the parameters $\theta_i$, $i=1, \ldots, 5$ against their theoretical true value. These distributions are obtained using the estimates obtained from the 1000 Monte Carlo replications. Figure \ref{fig1} indicates that all parameters are correctly estimated most of the times and, in particular, the parameter $\theta_3$ is often estimated as zero.

\subsection{An example of use in the problem of identification of the term structure of interest rates}
In this section we reanalyze the U.S. Interest Rates monthly data from 06/1964 to 12/1989 for a total of 307 observations. These data have been analyzed by many author including Nowman (1997), A\"\i t-Sahalia (1996), Yu and Phillips (2001) just to mention a few references. We do not pretend to give the definitive answer on the subject, but just to analyze the effect of the model selection via the LASSO in a real application.
The data used for this application were taken from the \textsf{R} package \texttt{Ecdat} by Croissant (2006).
The different authors all try to fit a version of the so called CKLS model (from Chan {\it et al.}, 1992) which is the solution $X_t$ of the following stochastic differential equation
$$
\de X_t = (\alpha + \beta X_t)\de t + \sigma X_t^\gamma \de W_t. 
$$
This model encompass several other models depending on the number of non-null parameters as Table \ref{tab1} shows. This makes clear why the model selection on the CKLS model is quite appealing.
\begin{center}
\fbox{\bf Table \ref{tab1} about here}
\end{center}
Our application of the LASSO method is reported in Table \ref{tab2} along with the results from Yu and Phillips (2001) just for comparison. 
\begin{center}
\fbox{\bf Table \ref{tab2} about here}
\end{center}
Although we have proven that asymptotically the LASSO provides consistent estimates with the oracle properties, for finite sample size this is not always the case as mentioned by several authors.
In this application,  we estimate the parameters using quasi-likelihood method (QMLE in the table) in the first stage, then set the penalties as in  \eqref{eq:penalty} and run the LASSO optimization.
We estimate the CKLS parameters via the LASSO using mild penalties (i.e. $\lambda_0=\gamma_0=1$ in \eqref{eq:penalty}) and strong penalties (i.e. $\lambda_0=\gamma_0=10$).
Very strong penalties suggest that the model does not contain the term $\beta$ and in both cases, the LASSO estimation suggest $\gamma=3/2$, therefore a model quite close to  Cox, Ingersoll and Ross (1980).
Being a shrinkage estimator, the LASSO estimates have very low standard error compared to the other cases.
As said, this application has been done to show the applicability of the LASSO method and we do not pretend to draw in depth conclusions from this empirical evidence which is out of our competence.

\section{Proofs}\label{sec4}

\begin{proof}[Proof of Theorem \ref{th1}]
 Following Fan and Li (2001), the existence of a consistent local minimizer is implied by that fact that for an arbitrarily small $\varepsilon > 0$, there exists a sufficiently large constant $C$, such that
\begin{equation}\label{prob}
\lim_{n\to \infty}P\left\{\inf_{z\in \mathbb{R}^{p+q}:|z|=C} \mathcal{F}(\theta_0+\varphi(n)^{1/2}z)> \mathcal{F}"(\theta_0)\right\}>1-\varepsilon,
\end{equation}
 with $z=(u,v)=(u_1,...,u_p,v_1,...,v_q)$. After some calculations, we obtain that
\begin{eqnarray*}
&&\mathcal{F}(\theta_0+\varphi(n)^{1/2}z)-\mathcal{F}(\theta_0)\\\\
&&=z\varphi(n)^{1/2}\ddot{\mathbb{H}}_n({\bf X}_n, \tilde\theta_n)\varphi(n)^{1/2}z'+2z\varphi(n)^{1/2}\ddot{\mathbb{H}}_n({\bf X}_n, \tilde\theta_n)\varphi(n)^{1/2}\varphi(n)^{-1/2}(\theta_0-\tilde\theta_n)'\\\\
&&\quad+n\Delta_n\left(\sum_{j=1}^p\lambda_{n,j}\left|\alpha_{0j}+\frac{u_j}{\sqrt{n\Delta_n}}\right|-\sum_{j=1}^p\lambda_{n,j}\left|\alpha_{0j}\right|\right)+n\left(\sum_{k=1}^q\gamma_{n,k}\left|\beta_{0k}+\frac{v_j}{\sqrt{n}}\right|-\sum_{j=1}^q\gamma_{n,k}\left|\beta_{0k}\right|\right)\\\\
&&=z\varphi(n)^{1/2}\ddot{\mathbb{H}}_n({\bf X}_n, \tilde\theta_n)\varphi(n)^{1/2}z'+2z\varphi(n)^{1/2}\ddot{\mathbb{H}}_n({\bf X}_n, \tilde\theta_n)\varphi(n)^{1/2}\varphi(n)^{-1/2}(\theta_0-\tilde\theta_n)'\\\\
&&\quad+n\Delta_n\left(\sum_{j=1}^p\lambda_{n,j}\left|\alpha_{0j}+\frac{u_j}{\sqrt{n\Delta_n}}\right|-\sum_{j=1}^{p_0}\lambda_{n,j}\left|\alpha_{0j}\right|\right)+n\left(\sum_{k=1}^q\gamma_{n,k}\left|\beta_{0k}+\frac{v_j}{\sqrt{n}}\right|-\sum_{j=1}^{q_0}\gamma_{n,k}\left|\beta_{0k}\right|\right)\\\\
&&\geq z\varphi(n)^{1/2}\ddot{\mathbb{H}}_n({\bf X}_n, \tilde\theta_n)\varphi(n)^{1/2}z'+2z\varphi(n)^{1/2}\ddot{\mathbb{H}}_n({\bf X}_n, \tilde\theta_n)\varphi(n)^{1/2}\varphi(n)^{-1/2}(\theta_0-\tilde\theta_n)'\\\\
&&\quad+n\Delta_n\sum_{j=1}^{p_0}\lambda_{n,j}\left(\left|\alpha_{0j}+\frac{u_j}{\sqrt{n\Delta_n}}\right|-\left|\alpha_{0j}\right|\right)+n\sum_{k=1}^{q_0}\gamma_{n,k}\left(\left|\beta_{0k}+\frac{v_j}{\sqrt{n}}\right|-\gamma_{n,k}\left|\beta_{0k}\right|\right)\\\\
&&\geq z\varphi(n)^{1/2}\ddot{\mathbb{H}}_n({\bf X}_n, \tilde\theta_n)\varphi(n)^{1/2}z'+2z\varphi(n)^{1/2}\ddot{\mathbb{H}}_n({\bf X}_n, \tilde\theta_n)\varphi(n)^{1/2}\varphi(n)^{-1/2}(\theta_0-\tilde\theta_n)'\\\\
&&\quad-\left[p_0(\sqrt{n\Delta_n}\mu_n)|u|+q_0(\sqrt{n}\nu_n)|v|\right]\\
&&=\Xi_1+\Xi_2-\Xi_3
\end{eqnarray*}
Now, it is clear that from the condition $\mathcal{C}_1$, one has that $\Xi_3=o_p(1)$. Furthermore, being $|z|=C$, $\Xi_1$ is uniformly larger than $\tau_{min}(\varphi(n)^{1/2}\ddot{\mathbb{H}}_n({\bf X}_n, \tilde\theta_n)\varphi(n)^{1/2})C^2$ and $$\tau_{min}(\varphi(n)^{1/2}\ddot{\mathbb{H}}_n({\bf X}_n, \tilde\theta_n)\varphi(n)^{1/2})C^2\stackrel{p}{\to}C^2\tau_{min}(\mathcal{I}(\theta_0))$$
where $\tau_{min}(A)$ is the minum eigenvalue of $A$.
We observe that 
$$|\varphi(n)^{1/2}\ddot{\mathbb{H}}_n({\bf X}_n, \tilde\theta_n)\varphi(n)^{1/2}\varphi(n)^{-1/2}(\theta_0-\tilde\theta_n)|=O_p(1)$$ 
and then $\Xi_2$ is bounded and linearly dependent on $C$.
Therefore, for $C$ sufficiently large, $\mathcal{F}(\theta_0+\varphi(n)^{1/2}z)-\mathcal{F}(\theta_0)$ dominates $\Xi_1+\Xi_2$ with arbitrarily large probability. This implies \eqref{prob} and the proof is completed by noticing that $\mathcal{F}(\theta)$ is striclty convex which implies that the local minimum is the global one.
\end{proof}

\begin{proof}[Proof of Theorem \ref{th2}] For $j=p_0+1,...,p$
$$\frac{1}{\sqrt{n\Delta_n}}\left.\frac{\partial \mathcal{F(\theta)}}{\partial \alpha_j}\right|_{\theta=\hat\theta_n}=2\frac{1}{n\Delta_n}\ddot{\mathbb{H}}_n^{(j)}({\bf X}_n, \tilde\theta_n)\sqrt{n\Delta_n}(\hat\theta_n-\tilde\theta_n)'+\frac{\lambda_{n,j}}{\sqrt{n\Delta_n}}{\rm sgn}(\hat\alpha_{n,j})$$
where $\ddot{\mathbb{H}}_n^{(j)}$ is the $j$-th row of $\ddot{\mathbb{H}}_n$.
The first term of the previous expression is $O_p(1)$, while $\frac{\lambda_{n,j}}{\sqrt{n\Delta_n}}\geq \frac{\kappa_n}{\sqrt{n\Delta_n}}\to \infty$. Since Theorem 1, $\hat\theta_n$ is a minimizer of $\mathcal{F}$, then necessarely, $P(\hat\alpha_{n,j}=0)\to 1$ (see Proof of Theorem 2, Wang and Leng, 2007). Similarly for the estimators of the coefficients $\beta_k,\,k=q_0+1,...,q $, we have that
$$\frac{1}{\sqrt{n}}\left.\frac{\partial \mathcal{F(\theta)}}{\partial \beta_k}\right|_{\theta=\hat\theta_n}=2\frac{1}{n}\ddot{\mathbb{H}}_n^{(k)}({\bf X}_n, \tilde\theta_n)\sqrt{n}(\hat\theta_n-\tilde\theta_n)'+\frac{\lambda_{n,j}}{\sqrt{n}}{\rm sgn}(\hat\beta_{n,j})$$
and by means the same arguments we get that $P(\hat\beta_{n,k}=0)\to 1$.
\end{proof}

\begin{proof}[Proof of Theorem \ref{th3}] Before starting the proof, it is necessary to introduce the following notations. Let 
\begin{itemize}
\item $\hat\Gamma_\alpha^{**}$ be the $p_0\times p_0$ matrix with elements $[\ddot{\mathbb{H}}_n]_{kj}$, $k,j=1,...,p_0$, 
\item $\hat\Gamma_\alpha^{*\circ}$ be the $p_0\times p-p_0$ matrix with elements $[\ddot{\mathbb{H}}_n]_{kj}$, $k=1,...,p_0$, $j=p_0+1,...,p$,
\item $\hat\Gamma_\alpha^{\circ\circ}$ be the $(p-p_0)\times (p-p_0)$ matrix with elements $[\ddot{\mathbb{H}}_n]_{kj}$, $k,j=p_0+1,...,p$,
\item $\hat\Gamma_\beta^{**}$ be the $p_0\times p_0$ matrix with elements $[\ddot{\mathbb{H}}_n]_{kj}$, $k,j=1,...,q_0$, 
\item $\hat\Gamma_\beta^{*\circ}$ be the $q_0\times q-q_0$ matrix with elements $[\ddot{\mathbb{H}}_n]_{kj}$, $k=1,...,q_0$, $j=q_0+1,...,q$,
\item $\hat\Gamma_\beta^{\circ\circ}$ be the $(q-q_0)\times (q-q_0)$ matrix with elements $[\ddot{\mathbb{H}}_n]_{kj}$, $k,j=q_0+1,...,q$,
\end{itemize}
where
$$
\frac{1}{n\Delta_n}
\left[\begin{array}{cc}\hat\Gamma_\alpha^{**}  & \hat\Gamma_\alpha^{*\circ}  \\
\hat\Gamma_\alpha^{*\circ}  & \hat\Gamma_\alpha^{\circ\circ} \end{array}\right]
 \stackrel{p}{\to}\Gamma_\alpha=
\left[\begin{array}{cc}\Gamma_\alpha^{**}  & \Gamma_\alpha^{*\circ}  \\
\Gamma_\alpha^{*\circ}  & \Gamma_\alpha^{\circ\circ} \end{array}\right]
$$
with
\begin{itemize}
\item $\Gamma_\alpha^{**}=[\mathcal I_b^{kj}(\alpha_0^*)]_{k,j}$, where $k,j=1,\ldots,p_0$,
\item $\Gamma_\alpha^{*\circ}=[\mathcal I_b^{kj}(\alpha_0^*)]_{k,j}$, where $k=1, \ldots,p_0; j=p_0+1,\ldots,p$,
\item $\Gamma_\alpha^{\circ\circ}=[\mathcal I_b^{kj}(\alpha_0^*)]_{k,j}$, where $k,j=p_0+1,\ldots,p$,
\end{itemize}
and
$$
\frac{1}{n}
\left[\begin{array}{cc}\hat\Gamma_\beta^{**}  & \hat\Gamma_\beta^{*\circ}  \\
\hat\Gamma_\beta^{*\circ}  & \hat\Gamma_\beta^{\circ\circ} \end{array}\right]
 \stackrel{p}{\to}\Gamma_\beta=
\left[\begin{array}{cc}\Gamma_\beta^{**}  & \Gamma_\beta^{*\circ}  \\
\Gamma_\beta^{*\circ}  & \Gamma_\beta^{\circ\circ} \end{array}\right]
$$
with
\begin{itemize}
\item $\Gamma_\beta^{**}=[\mathcal I_\sigma^{kj}(\beta_0^*)]_{k,j}$, where $k, j=1,\ldots,q_0$,
\item $\Gamma_\beta^{*\circ}=[\mathcal I_\sigma^{kj}(\beta_0^*)]_{k,j}$, where $k=1, \ldots,q_0; j=q_0+1,\ldots,q$,
\item $\Gamma_\beta^{\circ\circ}=[\mathcal I_\sigma^{kj}(\beta_0^*)]_{k,j}$, where $k,j=q_0+1,\ldots,q$.
\end{itemize}

From Theorem \ref{th2} follows that the estimator $\hat\theta_n$ globally minimizes of the following objective function
\begin{eqnarray*}
\mathcal{F}_0(\theta)&=&(\alpha^*-\tilde\alpha_n^*)\hat\Gamma_\alpha^{**}(\alpha^*-\tilde\alpha_n^*)'-2(\alpha^*-\tilde\alpha_n^*)\hat\Gamma_\alpha^{*\circ} \,(\tilde\alpha_n^{\circ})'+\tilde\alpha_n^\circ\,\hat\Gamma_\alpha^{\circ\circ}\,(\tilde\alpha_n^\circ)'+\sum_{j=1}^{p_0}\lambda_{n,j}|\alpha_j|\\
&&+(\beta^*-\tilde\beta_n^*)\hat\Gamma_\beta^{**}(\beta^*-\tilde\beta_n^*)'-2(\beta^*-\tilde\beta_n^*)\hat\Gamma_\beta^{*\circ}\,(\tilde\beta_n^\circ)'+\tilde\beta_n^\circ\,\hat\Gamma_\beta^{\circ\circ}\,(\tilde\beta_n^\circ)'+\sum_{k=1}^{q_0}\gamma_{n,k}|\beta_k|
\end{eqnarray*}
Hence, the following normal equations hold
 \begin{eqnarray}
 \label{eq:A4}
 0=\frac12 \left.\frac{\partial\mathcal{F}_0(\theta)}{\partial\alpha^*}\right|_{\alpha^*=\hat\alpha_n^*}=\hat\Gamma_\alpha^{**}(\hat\alpha_n^*-\tilde\alpha_n^*)'-\hat\Gamma_\alpha^{*\circ}\,(\tilde\alpha_n^\circ)'+A(\hat\alpha_n^*)
\end{eqnarray}
 \begin{eqnarray}
 \label{eq:A4A}
 0=\frac12 \left.\frac{\partial\mathcal{F}_0(\theta)}{\partial\beta^*}\right|_{\beta^*=\hat\beta_n^*}=\hat\Gamma_\beta^{**}(\hat\beta_n^*-\tilde\beta_n^*)'-\hat\Gamma_\beta^{*\circ}\,(\tilde\beta_n^\circ)'+B(\hat\beta_n^*)
\end{eqnarray}
where $A(\hat\alpha_n^*)$ and $B(\hat\beta_n^*)$ are respectively $p_0$ and $q_0$ vectors with $j$-th and $k$-th component given by $\frac12 \lambda_{n,j}{\rm sgn}(\hat\alpha_{n,j}^*)$ and $\frac12 \gamma_{n,k}{\rm sgn}(\hat\beta_{n,j}^*)$. From \eqref{eq:A4}, by simple calculations, we have that
 \begin{eqnarray*}
 \sqrt{n \Delta_n}(\hat\alpha^*_n - \alpha_0^*)&=&
 \sqrt{n \Delta_n}(\tilde\alpha_n^*-\alpha_0^*)+\left(\frac{1}{n \Delta_n}\hat\Gamma_\alpha^{**}\right)^{-1}\frac{1}{n \Delta_n}\hat\Gamma_\alpha^{*\circ} \sqrt{n \Delta_n}\tilde\alpha_n^\circ - (\hat\Gamma_\alpha^{**})^{-1}\sqrt{n \Delta_n}A(\hat\alpha_n^*)\\
&=& \sqrt{n \Delta_n}(\tilde\alpha_n^*-\alpha_0^*)+(\Gamma_\alpha^{**})^{-1}\Gamma_\alpha^{*\circ} \sqrt{n \Delta_n}\tilde\alpha_n^\circ +o_p(1)  
\end{eqnarray*}
being  $\sqrt{n \Delta_n}A(\hat\alpha_n^*)=o_p(1)$ by condition $\mathcal{C}_1$. Furthermore, by inverting the block matrix $\Gamma_\alpha$, we obtain that
$$\Gamma_\alpha^{-1}=
\left(\begin{array}{cc}(\Gamma_\alpha^{**})^{-1} &-(\Gamma_\alpha^{**})^{-1}\Gamma_\alpha^{*\circ}(\Gamma_\alpha^{\circ\circ})^{-1} \\\\
-(\Gamma_\alpha^{**})^{-1}\Gamma_\alpha^{*\circ}(\Gamma_\alpha^{\circ\circ})^{-1} \quad & (\Gamma_\alpha^{\circ\circ})^{-1} +(\Gamma_\alpha^{\circ\circ})^{-1} \Gamma_\alpha^{*\circ}(\Gamma_\alpha^{**})^{-1}\Gamma_\alpha^{*\circ}(\Gamma_\alpha^{\circ\circ})^{-1}
   \end{array}\right)
$$
where $(\Gamma_\alpha^{**})^{-1}=(\Gamma_\alpha^{**}-\Gamma_\alpha^{*\circ}(\Gamma_\alpha^{\circ\circ})^{-1}\Gamma_\alpha^{*\circ})^{-1}$
and then
$$(\Gamma_\alpha^{**})^{-1}\Gamma_\alpha^{*\circ}=(\Gamma_\alpha^{*\circ})^{-1}\Gamma_\alpha^{\circ\circ}.
$$
By condition $\mathcal{B}_2$ and the properties of the conditional multivariate Gaussian distribution, we derive that
$$\sqrt{n \Delta_n}(\tilde\alpha_n^*-\alpha_0^*)\stackrel{d}{\to} N(0, (\Gamma_\alpha^{**})^{-1} - (\Gamma_\alpha^{*\circ})^{-1}\Gamma_\alpha^{\circ\circ}(\Gamma_\alpha^{*\circ})^{-1})
$$
and
$$ (\Gamma_\alpha^{*\circ})^{-1}\Gamma_\alpha^{\circ\circ} \sqrt{n \Delta_n}\tilde\alpha_n^\circ \stackrel{d}{\to} N(0,  (\Gamma_\alpha^{*\circ})^{-1}\Gamma_\alpha^{\circ\circ}(\Gamma_\alpha^{*\circ})^{-1}).
$$
Thus $ \sqrt{n \Delta_n}(\hat\alpha^*_n - \alpha_0^*)$ converges to $N(0, (\Gamma_\alpha^{**})^{-1})$.
Similarly, from \eqref{eq:A4A} we obtain that
$$
 \sqrt{n}(\hat\beta^*_n - \beta_0^*) =
 \sqrt{n}(\tilde\beta_n^*-\beta_0^*)+(\Gamma_\beta^{**})^{-1}\Gamma_\beta^{*\circ} \sqrt{n}\tilde\beta_n^\circ +o_p(1)  
$$
with $\sqrt{n} B(\hat\beta_n^*)=o_p(1)$. Therefore,
$ \sqrt{n }(\hat\beta^*_n - \beta_0^*)$ converges to $N(0, (\Gamma_\beta^{**})^{-1})$. This concludes the proof.
\end{proof}

\begin{table}[t]
\begin{tabular}{l l ccc}
Reference & Model & $\alpha$ & $\beta$ & $\gamma$\\
\hline
Merton (1973) & $\de X_t = \alpha \de t + \sigma  \de W_t $& & 0 & 0\\
Vasicek (1977) & $\de X_t = (\alpha + \beta X_t)\de t + \sigma \de W_t$ & && 0\\
Cox, Ingersoll and Ross (1985) & $\de X_t = (\alpha + \beta X_t)\de t + \sigma \sqrt{X_t} \de W_t$&  & &$1/2$\\ 
Dothan (1978) & $\de X_t = \sigma X_t \de W_t $ & 0 & 0 &1\\
Geometric Brownian Motion &  $\de X_t = \beta X_t \de t + \sigma  X_t \de W_t $& 0&  & 1\\
Brennan and Schwartz (1980) & $\de X_t = (\alpha + \beta X_t)\de t + \sigma X_t \de W_t$ & &&1\\
Cox, Ingersoll and Ross (1980) & $\de X_t =  \sigma X_t^{3/2} \de W_t$&  0& 0&$3/2$\\ 
Constant Elasticity Variance & $\de X_t = \beta X_t\de t + \sigma X_t^\gamma \de W_t $ &0 &&\\
CKLS (1992) & $\de X_t = (\alpha + \beta X_t)\de t + \sigma X_t^\gamma \de W_t $
\end{tabular}
\caption{The family of one-factor short term interest rates models seen as special cases of the general CKLS model.}
\label{tab1}
\end{table}

\begin{table}
\begin{tabular}{l l cccc}
Model & Estimation Method & $\alpha$ & $\beta$ & $\sigma$ & $\gamma$\\
\hline
Vasicek & MLE & 4.1889 & -0.6072 & 0.8096 & -- \\
\\
CKLS & Nowman & 2.4272 & -0.3277 & 0.1741 & 1.3610\\
\\
CKLS & Exact Gaussian & 2.0069 & -0.3330  & 0.1741& 1.3610\\
&& (0.5216) & (0.0677)&&\\
\\
CKLS & QMLE &  2.0822 &-0.2756 &    0.1322 & 1.4392 \\
&&(0.9635) &  (0.1895)  & (0.0253)&  (0.1018)\\
\\
CKLS & QMLE + LASSO & 1.5435 &-0.1687  &0.1306& 1.4452\\
& with mild penalization& (0.6813) & (0.1340)  & (0.0179)&(0.0720)\\
\\
CKLS & QMLE + LASSO & 0.5412 &0.0001  &0.1178& 1.4944\\
&with strong penalization& (0.2076) & (0.0054)  & (0.0179)&(0.0720)
\end{tabular}
\caption{Model selection on the CKLS model for the U.S. interest rates data. Table taken from Yu and Phillips (2001) and updated with LASSO results. Standard errors in parenthesis when available.}
\label{tab2}
\end{table}

\begin{figure}
\includegraphics[width=\textwidth]{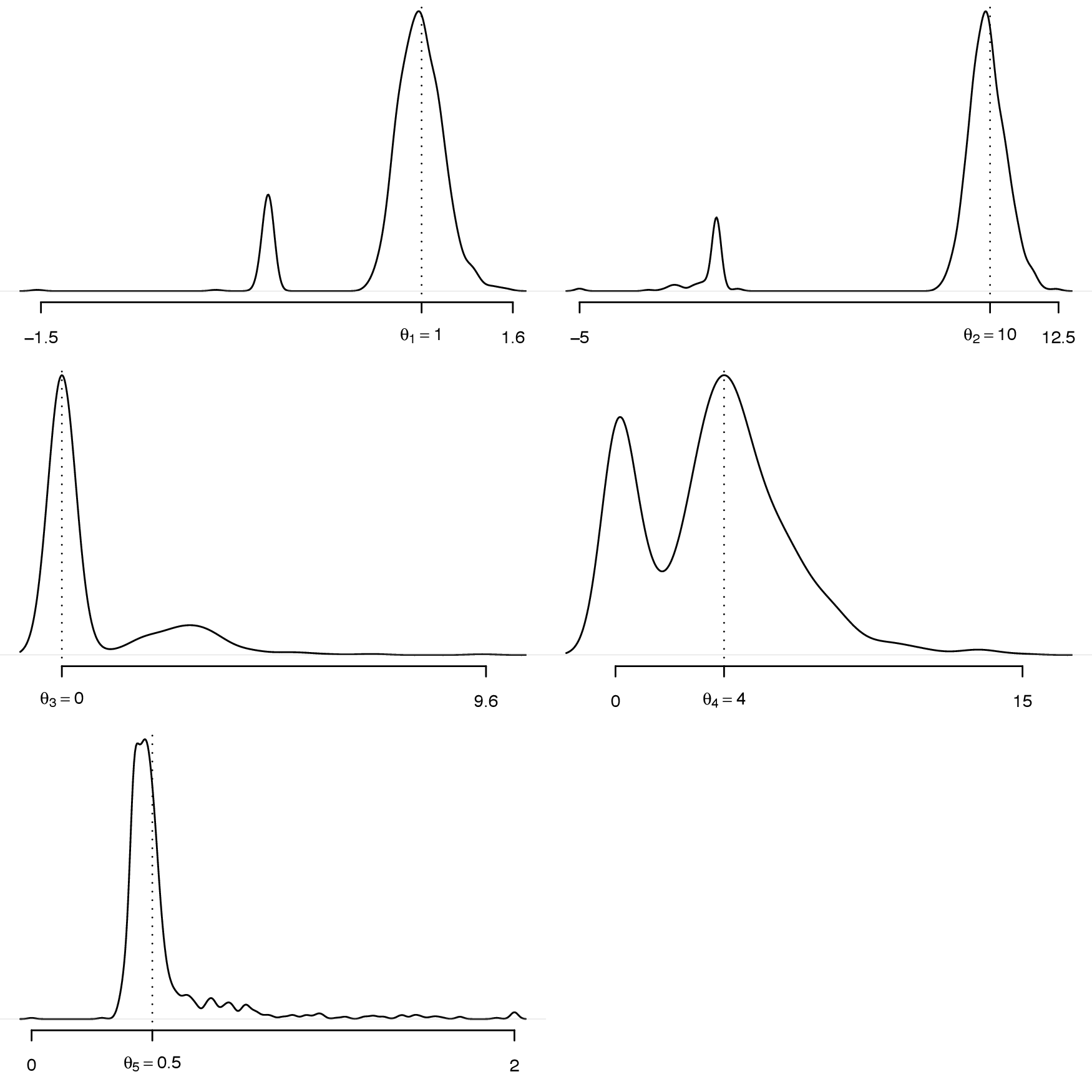}
\caption{Density estimation of the LASSO-type estimates of the parameters of the process $\de X_t = -\theta_1(X_t-\theta_2)\de t + (\theta_3+\theta_4 X_t)^{\theta_5}\de W_t$ over 1000 Monte Carlo replications. True values 
$(\theta_1=1, \theta_2=10, \theta_3=0, \theta_4=4, \theta_5=0.5)$ represented as vertical dotted lines.}
\label{fig1}
\end{figure}

\end{document}